\pdfoutput=1
\documentclass[12pt]{amsart}
\usepackage{amsmath, amsthm, verbatim}
\usepackage[all,cmtip]{xy}

\numberwithin{equation}{section}
\theoremstyle{plain}
\newtheorem{thm}[equation]{Theorem}
\newtheorem{prop}[equation]{Proposition}
\newtheorem{lem}[equation]{Lemma}
\theoremstyle{definition}
\newtheorem{defn}[equation]{Definition}
\newtheorem{exa}[equation]{Example}

\newcommand{\mf}{\phi}           \newcommand{\mm}{\mu}
          \newcommand{\me}{\eta}
\renewcommand{\mp}{\epsilon}     
\newcommand{\ii}{i_{\mathrm{I}}} 

\renewcommand{\a}{\alpha}        \newcommand{\g}{\gamma}
\renewcommand{\d}{\delta}        \newcommand{\e}{\epsilon}
\newcommand{\la}{\lambda}        \renewcommand{\t}{\theta}
\renewcommand{\H}{\mathbb{H}}    \newcommand{\N}{\mathbb{N}}
\newcommand{\R}{\mathbb{R}}      \newcommand{\bB}{\bar{B}}
\newcommand{\tg}{\tilde{\g}}     \newcommand{\mino}{\wedge}
\renewcommand{\S}{\mathbb{S}}
\newcommand{\bbd}[1]{\partial_{\mathrm{B}}#1}
\newcommand{\ebd}[1]{\partial_{\mathrm{E}}#1}
\newcommand{\ibd}[1]{\partial_{\mathrm{I}}#1}
\newcommand{\Gbd}[1]{\partial_{\mathrm{G}}#1}
\newcommand{\ds}{\displaystyle}
\renewcommand{\(}{\left(}        \renewcommand{\)}{\right)}
\newcommand{\ips}[3]{\left<#1,#2\right>_{#3}}

\begin{document}
\title{Natural maps between CAT(0) boundaries}
\author{S.M. Buckley and K. Falk}%
\address{Department of Mathematics and Statistics, National University of
Ireland Maynooth, Maynooth, Co. Kildare, Ireland}%
\email{stephen.buckley@maths.nuim.ie}%
\address{Universit\"at Bremen, FB 3 - Mathematik, Bibliothekstra{\ss}e 1,
28359 Bremen, Germany}%
\email{khf@math.uni-bremen.de}%
\begin{abstract}
It is shown that certain natural maps between the ideal, Gromov, and end
boundaries of a complete CAT(0) space can fail to be either injective or
surjective. Additionally the natural map from the Gromov boundary to the end
boundary of a complete CAT(-1) space can fail to be either injective or
surjective.
\end{abstract}
\thanks{Both authors thank the University of Bremen and National University of
Ireland Maynooth for hospitality and financial support of reciprocal visits
that enabled this research effort.}%
\maketitle

\renewcommand{\labelenumi}{(\alph{enumi})}

\section{Introduction}

In \cite{BF1}, a new class of metric spaces called {\it rough CAT(0)} spaces
are introduced, and their interior geometry is studied. The boundary
geometry is then defined and studied in \cite{BF2}. By {\it interior
geometry}, we mean the geometry of the space itself, whereas we use the term
{\it boundary} in the sense of a boundary at infinity. Specifically we
define a new notion of boundary at infinity for such spaces $X$ that we call
the {\it bouquet boundary} $\bbd X$.

Rough CAT(0) spaces include both of the well-known classes of CAT(0) spaces
and Gromov hyperbolic spaces, and it is proved in \cite{BF2} that $\bbd X$
coincides with the ideal boundary $\ibd X$ if $X$ is a complete CAT(0)
space, and it coincides with the Gromov boundary $\Gbd X$ if $X$ is a Gromov
hyperbolic space. With a view to proving that $\bbd X$ is nonempty when $X$
is a reasonable unbounded space, $\bbd X$ is also related to the end
boundary $\ebd X$, and in fact it is shown in \cite{BF2} that we have the
following commutative diagram of natural maps between these notions of
boundary at infinity:
$$
\xymatrix@C=6pc@R=3pc{
  \ibd X\;\ar@{^{(}->}[r]^{\ds{\ii}} &
  \bbd X
    \ar@{->}[d]_{\ds{\me}}
    \ar@{->}[r]^{\ds{\mm}} &
  \Gbd X \ar@{->}[dl]^{\ds{\mf}} \\
  & \ebd X & \\
}
$$

The results in \cite{BF2} are mainly positive: it is shown that $\ii$ is
injective, or bijective if $X$ is complete CAT(0), and that $\mm$ is
bijective if $X$ is Gromov hyperbolic. Furthermore conditions for $\me$ to
be surjective or injective are also given.

In this note, we instead concentrate on negative results in the context of
complete CAT(0) spaces. In this context, $\ii$ is a natural bijection, so we
will not even define the bouquet boundary: instead we identify it with the
well-known ideal boundary and study $\nu:=\mm\circ\ii$ instead of $\mm$.
Writing $\nu:=\mm\circ\ii$ and $\mp:=\me\circ\ii$, and omitting the
identification, we get the following simpler commutative diagram:
$$
\xymatrix@C=6pc@R=3pc{
  \ibd X\;
    \ar@{->}[d]_{\ds{\e}}
    \ar@{->}[r]^{\ds{\nu}} &
  \Gbd X \ar@{->}[dl]^{\ds{\mf}} \\
  \ebd X & \\
}
$$
It turns out that none of the natural maps in this last diagram are
necessarily bijective. In fact, we construct counterexamples in this note
that allow us to state the following theorem.

\begin{thm}\label{T:main1}\
There exists a complete CAT(0) space $X$ such that
\begin{enumerate}
\item $\nu:\ibd X\to\Gbd X$ is neither injective nor surjective;
\item $\mf:\Gbd X\to\ebd X$ is neither injective nor surjective;
\item $\mp:\ibd X\to\ebd X$ is neither injective nor surjective.
\end{enumerate}
Lastly, $\ibd Y$, $\Gbd Y$, and $\ebd Y$ may be empty even if $Y$ is an
unbounded complete CAT(0) space.
\end{thm}

Constructions were given in \cite[Section 3]{BK} for spaces in which $\nu$
fails to be injective or surjective. However those spaces were far from
being CAT(0). It is perhaps a little surprising that such counterexamples
also exist in the class of complete CAT(0) spaces.

By the results of \cite{BF2}, $\ibd X$, $\bbd X$, and $\Gbd X$ can all be
identified if $X$ is both Gromov hyperbolic and complete CAT(0), and so in
particular if $X$ is a complete CAT(-1) space. Thus $\nu$ is a bijection in
this case. However the map $\mf$ may still be badly behaved, as the
following result indicates.

\begin{thm}\label{T:main2}\
There exists a complete CAT(-1) space $X$ such that $\mf:\Gbd X\to\ebd X$ is
neither injective nor surjective. Also $\Gbd Y$ and $\ebd Y$ may be empty
even if $Y$ is an unbounded complete CAT(-1) space.
\end{thm}

After some preliminaries in Section~\ref{S:prelims}, we give counterexamples
and prove the above results in Section~\ref{S:examples}.

\section{Preliminaries}\label{S:prelims}

Throughout this section, we suppose $(X,d)$ is a metric space. We say that
$X$ is {\it proper} if every closed ball in $X$ is compact.

We write $A\mino B$ for the minimum of two numbers $A,B$.

A {\it $h$-short segment} from $x$ to $y$, $x,y\in X$, is a path of length
at most $d(x,y)+h$, $h\ge 0$. A {\it geodesic segment} is a $0$-short
segment. $X$ is a {\it length space} if there is a $h$-short segment between
each pair $x,y\in X$ for every $h>0$, and a {\it geodesic space} if there is
a geodesic segment between each pair $x,y\in X$.

A {\it geodesic ray} in $X$ is a path $\g:[0,\infty)\to X$ such that each
initial segment $\g|_{[0,t]}$ of $\g$ is a geodesic segment. The {\it ideal
boundary} $\ibd X$ of $X$ is the set of equivalence classes of geodesic rays
in $X$, where two geodesic rays $\g_1,\g_2$ are said to be equivalent if
$d(\tg_1(t),\tg_2(t))$ is uniformly bounded for all $t\ge 0$, where $\tg_i$
is the unit speed reparametrization of $\g_i$, $i=1,2$.

We refer the reader to \cite[Part II]{BH} for the theory of CAT($\kappa$)
spaces for $\kappa\in\R$. In particular, we note that a smooth Riemannian
manifold is CAT($\kappa$) if and only if it is simply connected and has
sectional curvature $\le\kappa$. Also the ideal boundary $\ibd X$ of a
complete CAT(0) space can be identified with the set of unit speed geodesic
rays from some origin $o\in X$ \cite[II.8.2]{BH}: this identification is
independent of the choice of origin. If $X$ is a simply connected smooth
Riemannian manifold of sectional curvature $\le 0$, then $\ibd X$ is
homeomorphic to $\S^{n-1}$.

As discussed in the introduction, we can identify the ideal boundary and the
bouquet boundary of \cite{BF2} in the context of complete CAT(0) spaces. The
fact that these are not the same in rough CAT(0) spaces, or even in
incomplete CAT(0) spaces, and that the bouquet boundary is better behaved
than the ideal boundary, is discussed in detail in \cite[Section~4]{BF2}, so
we will not discuss it further here. We simply identify these notions for
complete CAT(0) spaces.

We refer the reader to \cite{GH}, \cite{CDP}, \cite{Va}, or \cite[Part
III.H]{BH} for the theory of Gromov hyperbolic spaces. We use the
non-geodesic definition: a metric space $(X,d)$ is {\it $\d$-hyperbolic},
$\d\ge 0$, if
$$ \ips xzw\ge \ips xyw\mino \ips yzw - \d\,, \qquad x,y,z,w\in X\,, $$
where $\ips xzw$ is the Gromov product defined by
$$ 2\ips xyw = d(x,w) + d(y,w) - d(x,y)\,. $$

A {\it Gromov sequence} in a metric space $X$ is a sequence $(x_n)$ in $X$
such that $\ips{x_m}{x_n}{o}\to\infty$ as $m,n\to\infty$. If $x=(x_n)$ and
$y=(y_n)$ are two such sequences, we write $(x,y)\in E$ if
$\ips{x_m}{y_n}{o}\to\infty$ as $m,n\to\infty$. Then $E$ is a reflexive
symmetric relation on the set of Gromov sequences in $X$, so its transitive
closure, which we denote by $\sim$, is an equivalence relation on the set of
Gromov sequences in $X$. The {\it Gromov boundary} $\Gbd X$ is the set of
equivalence classes $[(x_n)]$ of Gromov sequences.

The relation $E$ above is an equivalence relation if $X$ is Gromov
hyperbolic, but this is not true in general metric spaces \cite[1.5]{BK}.
Gromov sequences and the Gromov boundary have mainly been considered in
Gromov hyperbolic spaces, but they have also been defined as above in
general metric spaces \cite{BK}. The Gromov boundary is independent of the
choice of basepoint $o$.

The natural map $\nu:\ibd X\to\Gbd X$ is induced by the map $f(\la)=(x_n)$
that takes a geodesic ray $\la$ parametrized by arclength to a sequence of
points $(x_n)$, where $x_n=\la(t_n)$ and $(t_n)$ is any sequence of numbers
tending to infinity.

A CAT(-1) space is both CAT(0) and Gromov hyperbolic; see II.1.12 and
III.H.1.2 of \cite{BH}. Hence by Theorems~4.20 and 5.15 of \cite{BF2}, we
can identify $\ibd X$ with $\Gbd X$ if $X$ is a complete CAT(-1) space. A
related result is Lemma III.H.3.1 of \cite{BH}, which says that we can
identify $\ibd X$ with $\Gbd X$ if $X$ is a proper geodesic Gromov
hyperbolic space.

By an {\it end} of a metric space $X$ (with basepoint $o$), we mean a
sequence $(U_n)$ of components of $X\setminus \bB_n$, where
$\bB_n=\overline{B(o,n)}$ for fixed $o\in X$ and $U_{n+1}\subset U_n$ for
all $n\in\N$. We do not require $\bB_n$ to be compact. We denote by $\ebd X$
the collection of ends of $X$ and call it the {\it end boundary of $X$}.

Ends with respect to different basepoints are compatible under set
inclusion: defining $U_n,V_n$ for all $n\in\N$ to be components of
$X\setminus\overline{B(o,n)}$ and $X\setminus\overline{B(o',n)}$,
respectively, it is clear that $U_n$ is a subset of a unique $V_m$ whenever
$n-m>d(o,o')$. This compatibility gives rise to a natural bijection between
ends with respect to different basepoints, allowing us to identify them and
treat the end boundary as being independent of the basepoint.

Suppose $X$ is a complete CAT(0) space. If we map a geodesic ray $\la$ from
$o$ parametrized by arclength to the end $(U_n)$, where $U_n$ is the
component of $X\setminus \bB_n$ containing $\la(n+1)$, then this map induces
the natural map $\mp:\ibd X\to\ebd X$; see \cite[Theorem~4.24]{BF2}. The
natural map $\mf:\Gbd X\to\ebd X$ is also induced by the map taking a Gromov
sequence $(x_n)$ to the unique end ``containing'' it, in the sense that for
each $m\in\N$ there exists $N\in\N$ such that $x_n$ lies in some component
$U_m$ of $X\setminus \bB_m$ for all $n\ge N$; see
\cite[Proposition~5.19]{BF2}.

\section{Examples}\label{S:examples}

Euclidean $\R^n$ for $n>1$ is all we need to consider to prove that $\nu$
and $\mp$ can fail to be injective. As is well known, the ideal boundary of
$\R^n$ (with cone topology attached) is homeomorphic to $\S^{n-1}$, the
sphere of dimension $n-1$. By contrast the end boundary of $\R^n$ is clearly
a singleton set, and it turns out that $\Gbd\R^n$ is also a singleton set.

Let us indicate how to prove this last fact. First since the natural map
$\nu$ exists, $\Gbd\R^n$ is nonempty. We now appeal to Theorem~2.2 of
\cite{BK} which states that $\nu$ is surjective if $X$ is a proper geodesic
space. In fact, as is clear from the proof of that result, $\nu$ is induced
by the map that takes a geodesic ray $\g:[0,\infty)\to X$ parametrized by
arclength to the Gromov sequence $(\g(t_n))_{n=1}^\infty$, where $(t_n)$ is
any sequence of non-negative numbers with limit infinity. Since the ideal
boundary of a complete CAT(0) space can be viewed as the set of geodesic
rays from any fixed origin, it follows that we get representatives of all
points in $\Gbd\R^2$ by considering only the Gromov sequences
$x^t:=(na_t)_{n=1}^\infty$, where $a_t=(\cos t,\sin t)\in\R^2$, $ t\in\R$. A
straightforward calculation, or an appeal to Lemma~\ref{L:angle a}, shows
that $(x^t,x^s)\in E$ for all pairs $t,s$, except when $|t-s|$ is an odd
multiple of $\pi$, i.e.~except when $x^t$ and $x^s$ are tending to infinity
in opposite directions. But in the exceptional case, we have
$(x^t,x^{t+\pi/2})\in E$ and $(x^{t+\pi/2},x^s)\in E$, so all Gromov
sequences are equivalent.

The exceptional case $\R^n$ for $n=1$ is easily analyzed: $\nu$ and $\mp$
are bijective, and the cardinalities of $\ibd X$, $\Gbd X$, and $\ebd X$ are
all $2$.

We now generalize this result to arbitrary Hilbert spaces of dimension
greater than $1$. Note that the above method of proof fails in this context
since infinite dimensional Hilbert spaces are not proper. Although these
infinite dimensional examples are not necessary ingredients in the proofs of
our main results, the method of proof will be useful for later examples that
are needed.

\begin{prop}\label{P:Gb bad1}
Suppose first that $X$ is a Hilbert space of dimension greater than $1$.
Then $\ibd X$ has cardinality at least that of the continuum, while $\Gbd X$
and $\ebd X$ are singleton sets. Thus the natural maps $\nu$ and $\mp$ are
not injective.
\end{prop}

Before proving Proposition~\ref{P:Gb bad1}, we first prove a simple but
useful lemma. In this lemma, and in the proof of Proposition~\ref{P:Gb
bad1}, $(\cdot,\cdot)$ is the inner product in a Hilbert space $X$,
$|\cdot|$ is the associated norm, and
$\angle(u,v)=\cos^{-1}\((u,v)/|u|\,|v|\)$ is the angle between two nonzero
vectors $u,v$ in $X$.

\begin{lem}\label{L:angle a}
Suppose $X$ is a Hilbert space, and that $u,v\in X\setminus\{0\}$ are such
that $\angle(u,v)\le\a$ for some $0<\a<\pi$. Then $\ips{u}{v}{0}\ge
k\,(|u|\mino|v|)$, where $k=k(\a)>0$.
\end{lem}

\begin{proof}
Writing $a:=|u-v|$, $b:=|u|$, and $c:=|v|$, we assume without loss of
generality that $b\le c$. By the cosine rule, $a^2 = b^2 + c^2 - 2bc\cos\a$.
Thus $2\ips{u}{v}{0}\ge cf(t)$, where $t=b/c$ and
$f(t)=t+1-\sqrt{1+t^2-2t\cos\a}$. Now $t\mapsto\sqrt{1+t^2+rt}$ is convex
for all $|r|\le 2$, so it follows by calculus that $f(t)\ge kt$, where
$k=f(1)>0$.
\end{proof}

\begin{proof}[Proof of Proposition~\ref{P:Gb bad1}]
Since $X$ is a complete CAT(0) space, its ideal boundary can be identified
with the set $R$ of unit speed geodesic rays from the origin. For a Hilbert
space, this latter set is naturally bijective to the sphere $S:=\partial
B(0,1)$ via the identification of $\g\in R$ with $\g(1)$. In particular, the
cardinality of $\ibd X$ is at least that of the continuum.

Now $(nu)_{n=1}^\infty$ is a Gromov sequence whenever $u\in S$, so certainly
$\Gbd X$ is nonempty. Suppose that $x=(x_n)$ and $y=(y_n)$ are Gromov
sequences. Pick $u_1,u_2,u_3\in S$ so that $\angle(u_i,u_j)\ge 2\pi/3$ for
each pair of distinct indices $i,j$; this can be even be done by picking
these points on a single great circle in $S$. For $i=1,2,3$, let $C_i$ be
the cone of points $u\in X\setminus\{0\}$ such that $\angle(u,u_i)<\pi/3$,
so that these three cones are pairwise disjoint. Thus, by taking
subsequences if necessary, we may assume that both of the sequences $x$ and
$y$ avoid one of these three cones $C_i$. Letting $z=(-nu_i)_{n=1}^\infty$,
and applying Lemma~\ref{L:angle a} with $\a=2\pi/3$, we see that $(x,z)$ and
$(z,y)$ both lie in $E$, and so $x\sim y$. Thus all Gromov sequences are
equivalent, as required.
\end{proof}

Finding examples where $\nu$ and $\mp$ fail to be surjective appears to be
more difficult than finding examples where they fail to be injective. The
key will be to consider a suitable metric subspace of the infinite
dimensional Hilbert space $\ell^2$ given by the following definition.

\begin{defn}\label{D:Hfs}
The {\it Hilbert flying saucer} is $X:=\bigcup_{i=1}^\infty
Y_i\subset\ell^2$, where $Y_i$ is the following closed disk of codimension
$i-1$ in $\ell^2$:
$$
Y_i := \{ x=(x_j)_{j=1}^\infty\;:\;
  \|x\|\le i,\; x_j=0 \text { for all } j<i \}\,.
$$
We attach to $X$ the induced length metric $d$. We next prove that $\nu$ can
fail to be surjective.
\end{defn}

\begin{thm}\label{T:Gb bad2}
The Hilbert flying saucer $X$ is a complete CAT(0) space. Moreover $\ibd X$
is empty, while $\Gbd X$ and $\ebd X$ are singleton sets. Thus the natural
maps $\nu$ and $\mp$ are not surjective.
\end{thm}

\begin{proof}
We define $o$ to be the origin in $\ell^2$ so that $o\in Y_i$ for all $i$.
Completeness is easy, since each $Y_i$ is closed in $\ell^2$ and a finite
number of the sets $Y_i$ cover any given compact set. To show that $X$ is
CAT(0), it suffices to show that $X_i:=\bigcup_{j=1}^i Y_j$ is CAT(0), where
again we attach the induced length metric. We establish this fact
inductively. Since any ball in $\ell^2$ is CAT(0), it follows that $Y_i$ is
CAT(0) for all $i$. In particular $X_1$ is CAT(0). Suppose now that $X_i$ is
CAT(0) for some $i$. Now $Z_i:=X_i\cap Y_{i+1}$ is a convex and closed
(hence complete) subset of $\ell^2$, and $X_{i+1}$ is obtained by gluing
$X_i$ and $Y_{i+1}$ along $Z_i$. But gluing two CAT(0) spaces along a pair
of isometric complete convex spaces gives another CAT(0) space
\cite[II.11.1]{BH}, so $X_{i+1}$ is CAT(0), completing the inductive step.
Thus $X$ is CAT(0).

Suppose for the sake of contradiction that $\g:[0,\infty)\to X$ is a
geodesic ray parametrized by arclength with $\g(0)=o$. Writing
$\g(1)=x=(x_i)$, we have $x_j\ne 0$ for some $j\in\N$. But
$\|\g(j+1)\|=j+1$, where $\|\cdot\|$ is the $\ell^2$-norm, so
$\g(j+1)=y=(y_i)$, where $y_i=0$ for all $i\le j$. But, as a subset of
$\ell^2$, $X$ is star-shaped with respect to $o$, so the $\ell^2$ line
segment is the unique $X$-geodesic from $o$ to $y$, and this does not pass
through $x$.

Using $o$ as the basepoint, it is clear that $X$ has a single end, so it
remains to prove that $\Gbd X$ is a singleton set. Let us denote by $e_i$
the unit vector in the $i$th coordinate direction. We denote distance in $X$
by $d(\cdot,\cdot)$ and the $\ell^2$ norm by $|\cdot|$. Although in general
we know only that $d(u,v)\ge|u-v|$, the $\ell^2$ line segment from $u$ to
$v$ is contained in $X$ if either $v=0$ or $|v|=|u|$ (in the latter case
because points on the line segment have norm no larger than $|u|$), and so
$d(u,v)=|u-v|$ in both of these cases.

Let $x_n=2^{n-1/2}(e_{2^n}+e_{2^n+1})$, and so $|x_n|=2^n$, $n\in\N$. Also
$\angle(x_i,x_j)=\pi/2$ for any pair of distinct indices $i,j$. It follows
from Lemma~\ref{L:angle a} that $(x_n)$ is a Gromov sequence in $\ell^2$. In
fact, it is also a Gromov sequence in $X$. To see this, note that if $i\le
j$ and we write $x_j'=2^{i-1/2}(e_{2^j}+e_{2^j+1})$, then
$d(x_i,x_j')=2^{i+1/2}$ and so $\ips{x_i}{x_j'}{0}=2^i k$, where
$k:=1-1/\sqrt{2}$. Now $d(x_i,x_j)\le d(x_i,x_j')+d(x_j',x_j)$ and
$d(0,x_j)=d(0,x_j')+d(x_j',x_j)$, so $\ips{x_i}{x_j}{0}\ge 2^i k$ for all
$i\le j$. Thus $x$ is a Gromov sequence in $X$, and $\Gbd X$ is nonempty.

It remains to prove that all Gromov sequences are equivalent, so suppose
$x=(x_n)$ and $y=(y_n)$ are a pair of Gromov sequences. Without loss of
generality, we assume that $x_n,y_n\ne 0$ for all $n\in\N$. The idea of this
proof is similar to that of Proposition~\ref{P:Gb bad1}: we pick a Gromov
sequence $z=(z_n)$, where $z_n=nu_n\in X$ and $u_n$ lies in the unit sphere
$S$ of $\ell^2$, such that the sequences $x$ and $y$ avoid a cone around
each of the points $z_n$, and it will then follow that both $(x,z)$ and
$(z,y)$ are elements of the relation $E$. Unlike the earlier proof, the
requirement that $z_n\in X$ means that $u_n$ must depend on $n$, and this
means that we will need to iterate a countable number of times the process
of taking subsequences of $x$ and $y$. When taking subsequences for the
$n$th time, we will insist that the first $n$ entries in the subsequences of
stage $n-1$ are retained at stage $n$: this ensures that the diagonal
subsequences associated with this process for $x$ and $y$ are subsequences
of the $n$th iterated subsequences of $x$ and $y$, respectively, for each
$n\in\N$.

For $n\in\N$, let $S_n$ be the intersection of $S$ with the plane generated
by $e_{2n-1}$ and $e_{2n}$: thus $(S_n)$ is a sequence of pairwise
orthogonal great circles on $S$. Pick five vectors $v_{1,i}$, $i=1,\dots,5$,
in $S_1\cup S_2$ such that $\angle(v_{1,i},v_{1,j})\ge\pi/2$ for all $1\le
i\le 5$: we could for instance pick four such vectors in $S_1$ and an
arbitrary vector in $S_2$. Thus the cones $C_{1,i}$ of points
$w\in\ell^2\setminus\{0\}$ such that $\angle(v_{1,i},w)<\pi/4$ are pairwise
disjoint, and so at least three of them must be disjoint from the set
$\{x_1,y_1\}$. Of these three, at least one contains infinitely many $x_n$
and infinitely many $y_n$. By taking subsequences $x^1=(x_n^1)$,
$y^1=(y_n^1)$, of $x$ and $y$, respectively, and letting $u_1$ be one of the
vectors $v_{1,i}$, we get that $\angle(u_1,w)\ge\pi/4$ whenever $w=x_n^1$ or
$w=y_n^1$ for any $n\in\N$. We assume, as we may, that $x_1^1=x_1$ and
$y_1^1=y_1$.

For the second stage, we pick seven vectors $v_{2,i}$, $i=1,\dots,7$, in
$S_3\cup S_4$ such that $\angle(v_{2,i},v_{2,j})\ge\pi/2$ for all $1\le
i<j\le 7$: we could for instance pick four such vectors in $S_3$ and three
in $S_4$. At least three of the seven associated cones fail to intersect the
set $\{x_1^1,x_2^2,y_1^1,y_2^1\}$. It follows that by taking subsequences
$x^2=(x_n^2)$, $y^2=(y_n^2)$, of $x^1$ and $y^1$, respectively, and letting
$u_2$ be one of the vectors $v_{2,i}$, we get that $\angle(u_2,w)\ge\pi/4$
whenever $w=x_n^2$ or $w=y_n^2$ for any $n\in\N$. We assume, as we may, that
$x_n^2=x_n^1$ and $y_n^2=y_n^1$ for $n=1,2$.

We proceed in this manner, picking vectors $v_{m,i}$, $i=1,\dots,2m+3$, at
the $m$th stage from the next few circles $S_l$ that we have not used yet in
this construction, in such a way that $\angle(v_{m,i},v_{m,j})\ge\pi/2$ for
all $1\le i<j\le 2m+3$. We use as many circles as are needed to ensure that
this can be done: for $m\in\{2p-1,2p\}$, $p\in\N$, it suffices to use $p+1$
circles. Thus for $m=3$, we pick nine vectors from $S_5\cup S_6\cup S_7$,
for $m=4$ we pick eleven vectors from $S_8\cup S_9\cup S_{10}$, etc.
Carrying out the construction as before, we get a unit vector $u_m\in S_M$
for some $M\ge m$, and subsequences $x^m=(x_n^m)$, $y^m=(y_n^m)$ of
$x^{m-1}$ and $y^{m-1}$, respectively, such that $\angle(u_m,w)\ge\pi/4$
whenever $w=x_n^m$ or $w=y_n^m$ for some $n\in\N$, and such that
$x_n^m=x_n^{m-1}$ and $y_n^m=y_n^{m-1}$ whenever $n\le m$.

Defining $x_n'=x_n^n$, $y_n'=y_n^n$ for $n\in\N$, we get subsequences
$x'=(x_n')$, $y'=(y_n')$ of $x,y$, respectively, and a sequence $u=(u_n)$ on
$S$ such that $\angle(u_n,w)\ge\pi/4$ whenever $w\in\{x_m,y_m\}$ and
$n,m\in\N$, and such that $\angle(u_m,u_n)=\pi/2$ whenever $m\ne n$. Letting
$z_n=nu_n$, it follows from the fact that $u_m\in S_M$ for some $M\ge m$
that $z_n\in X$ for all $n\in\N$. Writing $z=(z_n)$, and arguing as we did
in the proof that $\Gbd X$ is nonempty, it follows that $z$ is a Gromov
sequence, and that both $(x,z)$ and $(z,y)$ lie in $E$. We leave the details
to the reader.
\end{proof}

It is easy to find a complete CAT(-1) space $X$ in which $\mf:\Gbd X\to\ebd
X$ is not injective. Indeed the hyperbolic plane $X=\H^2$ is one such
example. Since we can identify $\ibd X$ with $\Gbd X$ in this case (because
$X$ is complete CAT(-1), or alternatively because $X$ is a proper geodesic
Gromov hyperbolic space: see Section~\ref{S:prelims}), $\Gbd X$ can be
identified with the unit circle, and so its cardinality is that of the
continuum. On the other hand, it is clear that $\ebd X$ is a singleton set.
To prove that $\mf$ may fail to be surjective even if $X$ is complete
CAT(-1), we need to work a bit harder.

\begin{exa}\label{X:hyperbolic saucer}
Let $X$ be the Hilbert flying saucer, but let us replace our original metric
$d$ by the conformally distorted length metric $d'$ given infinitesimally at
a point with polar coordinates $(t,\t)\in[0,\infty)\times\partial B(0,1)$ by
$ds$, where $ds^2=dt^2+\sinh^2(t)d\t^2$. Then $(X,d')$ has a single end,
since $d'$-balls around $0$ coincide with $d$-balls.

Each of the sets $Y_i$ in Definition~\ref{D:Hfs} is CAT(-1) with respect to
$d'$: in fact any geodesic triangle in $Y_i$ is contained in an isometric
copy of a hyperbolic plane. Since we obtain $(X,d')$ by gluing a succession
of spaces $X_i$ and $Y_{i+1}$ along a pair of isometric complete convex
spaces, the resulting space $(X,d')$ is CAT(-1). It is also clearly
complete. Thus $\ibd X$ can be identified with $\Gbd X$. But the $d$- and
$d'$-geodesic paths from $0$ to $x\in X$ coincide as sets, so again $\ibd X$
is empty.
\end{exa}

We are now ready to prove our main theorems.

\begin{proof}[Proof of Theorem~\ref{T:main1}]
If we glue a pair of disjoint complete CAT(0) spaces $X$ and $Y$ by
identifying a single point in $X$ with another point in $Y$, then we get
another CAT(0) space which we denote $X+Y$ according to the basic gluing
theorem II.11.1 of \cite{BH}. It is also easy to see that $X+Y$ is complete
and that the ideal, Gromov, and end boundaries of $X+Y$ can be identified
with the disjoint union of the corresponding boundaries in $X$ and $Y$. The
maps $\nu$, $\mf$, and $\mp$ for $X+Y$ are also obtained by taking the
``disjoint union'' of the corresponding maps for $X$ and $Y$,
e.g.~$\nu_{X+Y}$ is defined by $\nu_{X+Y}(x)=\nu_X(x)$ for all $x\in\ibd X$
and $\nu_{X+Y}(y)=\nu_X(y)$ for all $y\in\ibd Y$. It follows that if we have
separate spaces where each of the maps $\nu$, $\mf$, and $\mp$ fails to be
injective or surjective, then by gluing all of these spaces at a single
point, we get a space where all three of these maps fails to be both
injective and surjective.

Now $\nu$ and $\mp$ fail to be injective in a Hilbert space of dimension
larger than $1$, and they fail to be surjective in the Hilbert flying
saucer, according to Proposition~\ref{P:Gb bad1} and Theorem~\ref{T:Gb
bad2}. As for $\mf$, it fails to be injective in the hyperbolic plane, and
it fails to be surjective in a hyperbolic version of the Hilbert flying
saucer (Example~\ref{X:hyperbolic saucer}). Since all of these spaces are
complete CAT(0) spaces, we can glue them to get a complete CAT(0) space that
fails to have any of these injectivity or surjectivity properties.

Finally, let $X$ be the subset of $\R^2$ consisting of the union of the line
segments from $(0,0)$ to $(n,1)$ for all $n\in\N$, and attach the Euclidean
length metric to $X$. Then $X$ is an unbounded tree but it has no end, so
$\ebd X$, $\Gbd X$, and $\ibd X$ are all empty.
\end{proof}

The proof of Theorem~\ref{T:main2} is very similar. The hyperbolic version
of the Hilbert flying saucer in Example~\ref{X:hyperbolic saucer} is
complete CAT(-1), so when we glue it at a single point to the hyperbolic
plane, the resulting space is also complete CAT(-1). By the properties of
the individual space, we see that $\mf$ for the glued space fails to be
either injective or surjective. Finally the example in the last paragraph of
the previous proof is CAT($-\infty$), so it also works for
Theorem~\ref{T:main2}.


\end{document}